\documentclass[11pt]{amsart}

\usepackage[margin=1.1in]{geometry}
\usepackage{amsmath,amssymb,amsfonts,amsthm,mathtools}
\usepackage{mathrsfs}
\usepackage{enumitem}
\usepackage[colorlinks=true,citecolor=blue,linkcolor=blue,urlcolor=blue]{hyperref}

\numberwithin{equation}{section}

\newtheorem{theorem}{Theorem}[section]
\newtheorem{proposition}[theorem]{Proposition}
\newtheorem{lemma}[theorem]{Lemma}

\newtheorem{conjecture}[theorem]{Conjecture}
\theoremstyle{definition}

\newcommand{\R}{\mathbb R}

\newcommand{\Sph}{\mathbb S}
\newcommand{\eps}{\varepsilon}

\newcommand{\Euc}{\mathrm{Euc}}
\newcommand{\tr}{\operatorname{tr}}
\newcommand{\divE}{\operatorname{div}}
\newcommand{\Ann}{\operatorname{Ann}}
\newcommand{\spt}{\operatorname{spt}}
\newcommand{\loc}{\mathrm{loc}}
\newcommand{\fint}{\mathop{\!\!\int\!\!\!\!-}\nolimits}

\newcommand{\Dfun}{\mathscr D}

\newcommand{\Efun}{\mathcal E}

\newcommand{\rhoE}{\rho}
\newcommand{\LD}{\mathcal L_D}
\newcommand{\II}{\mathrm{II}}
\newcommand{\ip}[2]{\langle #1,#2\rangle}
\newcommand{\norm}[1]{\left\lVert #1\right\rVert}
\newcommand{\abs}[1]{\left\lvert #1\right\rvert}

\newcommand{\Crit}{\operatorname{Crit}}

\title[Endpoint $C^0$ rigidity]{Gromov's Euclidean Endpoint $C^0$ Rigidity for the Positive Mass Theorem}
\author{Jiangcheng You and Heng Zhang}
\address{School of Mathematical Sciences, University of Science and Technology of China, Hefei, China}
\email{yjcmp@mail.ustc.edu.cn}
\email{hengz@mail.ustc.edu.cn}

\date{\today}

\begin{document}

\begin{abstract}
We prove Gromov's Euclidean endpoint $C^0$ rigidity conjecture.  Let $g$ be a smooth complete metric on $\R^3$ with non-negative scalar curvature.  If
$$
        |g-g_{\Euc}|=o(r^{-1}),\qquad r=|x|\to\infty,
$$
then $(\R^3,g)$ is isometric to Euclidean space.
\end{abstract}

\maketitle

\tableofcontents

\section{Introduction}

The positive mass theorem is a basic bridge between scalar curvature and the geometry of an end.  In its classical three-dimensional asymptotically flat form, it asserts that a complete asymptotically flat manifold with non-negative scalar curvature has non-negative ADM mass, and that the zero-mass case is rigid: the manifold is isometric to Euclidean space.  The ADM mass was introduced by Arnowitt--Deser--Misner \cite{ADM}; the theorem was proved by Schoen--Yau \cite{SY1,SY2} and, by a spinorial argument, by Witten \cite{Witten}.  Under the standard asymptotically flat hypotheses, the coordinate invariance of the ADM mass was established by Bartnik \cite{BartnikMass}.  We will also use the Green-function formulation and rigidity mechanism of Agostiniani--Mazzieri--Oronzio \cite{AMO}.

Scalar curvature is a second-order expression in the metric.  Thus the usual asymptotically flat framework imposes differentiable decay assumptions strong enough to make the ADM surface integral meaningful.  At the same time, a growing body of work shows that lower scalar-curvature bounds have a notable $C^0$ character.  Gromov proved that non-negative scalar curvature is preserved under $C^0$ convergence of smooth manifolds to a smooth limit \cite{GromovDiracPlateau,GromovFourLectures}; related approaches include Bamler's Ricci-flow proof \cite{Bamler}, Burkhardt-Guim's Ricci-flow definition for $C^0$ metrics \cite{BurkhardtGuimPointwise}, and recent quantitative or weak-limit results using smoothing, harmonic functions, and $\mu$-bubbles \cite{LeeSmoothing,MYquantification,MYweaklimits}.  The prism and polyhedron rigidity theorems of Li and Brendle \cite{LiPolyhedron,BrendlePolytopes} provide another geometric expression of this same low-regularity principle.

This $C^0$ viewpoint naturally leads back to the rigidity statement in the positive mass theorem.  If non-negative scalar curvature can be detected through such weak metric information, one may ask whether the rigidity conclusion survives when the usual asymptotically flat expansion at infinity is replaced by a purely $C^0$ rate of convergence to the Euclidean metric.  Gromov formulated this problem in \cite[Section~3.11]{GromovFourLectures}. More precisely,

\begin{conjecture}[Gromov's Euclidean endpoint $C^0$ rigidity]\label{conj:gromov}
Assume that $g$ is a smooth complete metric on $\R^3$ with non-negative scalar curvature.  If
$$
        |g-g_{\Euc}|=o(r^{-1}),\qquad r=|x|\to\infty,
$$
then $(\R^3,g)$ is isometric to Euclidean space.    
\end{conjecture}

In this work, we show this conjecture is true.
\begin{theorem}\label{thm:intro-main}
Conjecture \ref{conj:gromov} holds true.
\end{theorem}

We remark that the endpoint decay in Theorem~\ref{thm:intro-main} is sharp in the following sense.  For every
$c>0$, there exist smooth non-flat metrics $g$ on $\R^3$ with non-negative scalar curvature such that
$$
        g=\bigl(1+c|x|^{-1}+O(|x|^{-2})\bigr)g_{\Euc}
$$
near infinity.  One way to obtain such examples is to start from a small perturbation of the round
metric on $\mathbb S^3$ and then conformally blow up a point using the Green function of the
conformal Laplacian; see \cite[Remark~4]{MYrigidity}.  Thus the assumption
$o(|x|^{-1})$ cannot in general be weakened to an $O(|x|^{-1})$ condition.

\subsection{Relation with the isoperimetric PMT of Benatti--Fogagnolo--Mazzieri}

Benatti, Fogagnolo, and Mazzieri \cite{BFM} proved an isoperimetric version of the positive mass theorem. Here \(\mathfrak m_{\mathrm{iso}}\) denotes Huisken's isoperimetric mass \cite{HuiskenIsoMass}, in the exhaustion formulation of \cite[Definition~1.2]{BFM}: it is the supremum over exhaustions of the limsup of the quasi-local isoperimetric deficit
\[
        \mathfrak m_{\mathrm{iso}}(\Omega)
        =
        \frac{2}{|\partial\Omega|}
        \left(
        |\Omega|-\frac{|\partial\Omega|^{3/2}}{6\sqrt{\pi}}
        \right).
\] In the case relevant to the present paper, \cite[Theorem~2.14]{BFM} asserts that if a complete strongly $1$-nonparabolic Riemannian $3$-manifold with non-negative scalar curvature satisfies the topological hypotheses in their Lemma~2.11(2), then
$
        0\le \mathfrak m_{\mathrm{iso}},
$
and equality holds if and only if the manifold is isometric to the flat $\R^3$. 

This gives a quick route to the non-endpoint version of Conjecture~\ref{conj:gromov}.  Suppose, for instance, that
$$
        |g-g_{\Euc}|=O(r^{-1-\tau}),\qquad \tau>0.
$$
Then $g$ is $C^0_\tau$-asymptotically Schwarzschildian with Schwarzschild mass parameter $0$ in the terminology of \cite{BFM}.  Their Theorem~1.5 gives
$
        \mathfrak m_{\mathrm{iso}}=0.
$
The equality case of  \cite[Theorem~2.14]{BFM} then implies that $(\R^3,g)$ is isometric to Euclidean space.

However, the endpoint assumption $$ |g-g_{\mathrm{Euc}}|=o(r^{-1}) $$ is not covered
by these theorems as stated.  In the proof of \cite[Theorem~1.5]{BFM}, the
comparison between the quasi-local isoperimetric mass of the weak IMCF
regions and the Schwarzschild model uses the quantitative
\(O(r^{-1-\tau})\) error, with a fixed positive exponent \(\tau\), to make
the corresponding error terms vanish at infinity.  At the endpoint one no
longer has such a fixed power decay available from the stated hypotheses.
Moreover, since $\mathfrak m_{\mathrm{iso}}$ is a supremum over all
exhaustions, the fact that the coordinate-ball quasi-local masses tend to
zero under \(o(r^{-1})\) does not by itself imply
\(\mathfrak m_{\mathrm{iso}}\le0\).

\subsection{Relation with the work of Mazurowski--Yao}
Recently, Mazurowski and Yao \cite{MYrigidity} gave a new proof Theorem~\ref{thm:intro-main} under the hypothesis
$$
        |g-g_{\Euc}|=O(r^{-1-\tau}),\qquad \tau>0.
$$
Their argument has two main components.

First, let $u$ be the normalized Green function for the divergence-form operator
$$
        Lf=\divE(A\nabla f),\qquad A=\sqrt{\det(g_{kl})}\,g^{-1}=I-B.
$$
Away from the pole,
$$
        \Delta u=\divE(B\nabla u).
$$
Set
$$
        X=B\nabla u.
$$
Under the faster decay $B=O(r^{-1-\tau})$, the Caccioppoli estimate for $u$ gives $X\in L^1(\R^3)$.  Mazurowski--Yao then define 
$$
        P_\infty=\frac1{4\pi}\int_{\R^3}X
$$
and obtain the rescaled Green expansion: on every fixed annulus,
$$
        R\left(Ru(Ry)-|y|^{-1}\right)
        \longrightarrow \frac{\langle b+P_\infty,y\rangle}{|y|^3}
        \quad\text{in } W^{1,p}.
$$

Second, they insert this expansion into an averaged Agostiniani--Mazzieri--Oronzio functional $D(a)$, as in \cite[Section~3.1]{MYrigidity}.  The $F$-functional of Agostiniani--Mazzieri--Oronzio is monotone when $R_g\ge0$ by \cite[Theorem~1.1]{AMO}, equivalently in the normalization of Mazurowski--Yao by \cite[Appendix~A, Proposition~22]{MYrigidity}.  The averaged functional $D$ inherits
$$
        D(a)\ge0,
        \qquad aD(a)\text{ is non-decreasing}.
$$
Mazurowski--Yao show that their Green expansion implies $aD(a)\to0$.  Hence $D\equiv0$, then $F\equiv0$, and the rigidity theorem for the $F$-functional gives flatness.

The obstruction at the endpoint is exactly the loss of $X\in L^1$.  From $B=o(r^{-1})$ one only obtains the dyadic estimate
$$
        m_j:=\int_{2^j\le |x|\le 2^{j+1}} |X(x)|\,dx\to0.
$$
The full integral of $X$ may diverge.  Therefore $P_\infty$ used in \cite{MYrigidity} may not exist.

\subsection{Proof strategy}
The proof below keeps the $D/F$-functional mechanism of Mazurowski--Yao \cite[Sections~3.1--3.3 and Appendix~A]{MYrigidity} but replaces the fixed $P_\infty$ analysis by a scale-dependent one.  The first new point is that the endpoint dyadic estimate is still strong enough to define, for each annular scale $R$,
$$
        P_R^\sigma=\frac1{4\pi}\int_{|x|\le \sigma R}X(x)\,dx.
$$
Although $P_R^\sigma$ need not converge, it satisfies
$$
        |P_R^\sigma|=o(\log R).
$$
We prove that the Green function has the following expansion
$$
        R\left(Ru(Ry)-|y|^{-1}\right)
        =\frac{\langle b+P_R^\sigma,y\rangle}{|y|^3}
        +o_{W^{1,p}}(1)
$$
on each fixed annulus, along dyadic scales. The proof requires a weighted convolution construction for $\Delta w=\divE X$, a three-region dyadic estimate for $P_R^\sigma$, and a perturbative elliptic upgrade from $L^q$ to $W^{1,p}$ on fixed annuli.  The last estimate is proved in Appendix~\ref{app:elliptic} directly from the constant-coefficient Calderon--Zygmund estimate and a hole-filling iteration.

The second new point is that the $D$-functional satisfies
$$
        \LD\left(0,\frac{\langle \mathbf q,y\rangle}{|y|^3}\right)=0
        \qquad\text{for all }\mathbf q\in\R^3,
$$
where $\mathcal L_D$ denotes the linearization of $D$. This identity is exact for each $\mathbf q$.  Hence it remains valid with the moving coefficient $\mathbf q=Q_R^\sigma=b+P_R^\sigma$.  To use this observation at the endpoint we need a quadratic, not merely Fr\'echet, expansion of $D$ at the Euclidean model:
$$
        \mathscr D(\delta+k,\rho+v)
        =\mathscr D(\delta,\rho)+\LD(k,v)
        +O\bigl((\|k\|_{C^0}+\|v\|_{W^{1,p}})^2\bigr).
$$
Since $|Q_R^\sigma|=o(\log R)$, the quadratic contribution satisfies
$$
        R\left(\frac{|Q_R^\sigma|}{R}\right)^2=o(1).
$$
This yields $R D(R)\to0$ along dyadic scales.  The monotonicity of $R D(R)$ then forces $D\equiv0$.

Finally, Appendix~\ref{app:FD-endpoint} records the precise endpoint use of the $F/D$ machinery.  The monotonicity formula is the Agostiniani--Mazzieri--Oronzio formula \cite[Theorem~1.1, formula~(1.3)]{AMO}, written in the present normalization.  The rigidity step is invoked as the rigidity statement of \cite[Corollary~1.3]{AMO}, equivalently the equality case following their monotonicity theorem \cite[Theorem~1.1]{AMO}.  The appendix verifies the hypotheses needed to invoke this theorem in the present setting.  It also records properness and connectedness of regular Green levels under the endpoint assumptions.

\subsection{Organization}
Section~\ref{sec:setup} proves the basic annular estimates for $X=B\nabla u$.  Section~\ref{sec:potential} constructs the weighted convolution and proves the $L^q$ expansion depending on $P_R^\sigma$.  Section~\ref{sec:green} upgrades this to the $W^{1,p}$ Green expansion.  Sections~\ref{sec:D}--\ref{sec:conclusion} analyze the $D$-functional and prove Theorem~\ref{thm:intro-main}.  Appendix~\ref{app:cesaro} contains elementary dyadic summation facts, Appendix~\ref{app:elliptic} proves the perturbative elliptic estimate, and Appendix~\ref{app:FD-endpoint} records the endpoint $F/D$ facts and the precise AMO rigidity input.

\subsection*{Acknowledgment} We thank Liam Mazurowski and Xuan Yao for discussions and for sharing relevant literature. We also express our gratitude to Shiping Liu and Hao Yin for their suggestions.

\section{Basic setup and annular control}
\label{sec:setup}

Throughout the paper, unless a subscript $g$, $g_R$, or $g_a$ is displayed, the symbols
$$
        \nabla,\quad \Delta,\quad \divE,\quad \ip{\cdot}{\cdot},\quad |\cdot|,
        \quad dx,\quad B_R
$$
refer to the Euclidean structure in the fixed coordinates on $\R^3$.  We write $\delta=g_{\Euc}$ for the Euclidean metric and $I=(\delta^{ij})$ for the identity matrix acting on coordinate one-forms.  The notation $\nabla_g$, $dv_g$, $dA_g$, and $\operatorname{div}_g$ refers to the Riemannian metric $g$.

Let $g$ be a smooth complete Riemannian metric on $\R^3$ satisfying
\begin{equation}\label{eq:c0-decay}
        \abs{g-g_{\Euc}}=o(r^{-1}),\qquad r=\abs{x}.
\end{equation}
Define
$$
        A=(a^{ij}),\qquad a^{ij}=\sqrt{\det(g_{kl})}\,g^{ij}.
$$
Then $A$ is uniformly elliptic and bounded.  Moreover,
\begin{equation}\label{eq:A-I-B}
        A=I-B,
        \qquad \abs{B(x)}=o(\abs{x}^{-1}).
\end{equation}
Let
$$
        Lf=\divE(A\nabla f)=\sqrt{\det(g_{kl})}\,\Delta_g f.
$$
Let $u$ be the positive Green function for $L$ with pole at the origin, normalized so that the $g$-flux of $u$ across each regular level set is $4\pi$.  By the Gr\"uter--Widman Green function estimates for uniformly elliptic divergence-form operators \cite[Theorem~1.1]{GW},
\begin{equation}\label{eq:green-two-sided}
        \frac{c_0}{\abs{x}}\le u(x)\le \frac{C_0}{\abs{x}}
        \qquad \text{for }\abs{x}\ge1.
\end{equation}

For $R>0$ define
$$
        \Ann(R)=B_{8R}\setminus B_{R/8}
$$
and the endpoint modulus
\begin{equation}\label{eq:eta-R}
        \eta(R)=\sup_{x\in\Ann(R)}R\abs{B(x)}.
\end{equation}
Then
\begin{equation}\label{eq:eta-zero}
        \eta(R)\to0.
\end{equation}

\begin{lemma}[Caccioppoli estimate]\label{lem:caccioppoli}
There is a constant $C$ independent of $R\ge1$ such that
\begin{equation}\label{eq:caccioppoli}
        \int_{\Ann(R)}\abs{\nabla u}^2\,dx\le \frac{C}{R}.
\end{equation}
\end{lemma}

\begin{proof}
This is the standard Caccioppoli argument for uniformly elliptic divergence-form equations; see for example \cite[Chapter~8]{GT}.  Choose a cutoff $\chi_R$ satisfying $\chi_R\equiv1$ on $\Ann(R)$, supported in $B_{16R}\setminus B_{R/16}$, and $\abs{\nabla\chi_R}\le C/R$.  Testing $Lu=0$ against $\chi_R^2u$ gives
$$
        0=\int \ip{A\nabla u}{\nabla(\chi_R^2u)}\,dx.
$$
Uniform ellipticity, boundedness of $A$, and Young's inequality imply
$$
        \int \chi_R^2\abs{\nabla u}^2\,dx
        \le C\int u^2\abs{\nabla\chi_R}^2\,dx.
$$
Using \eqref{eq:green-two-sided} on the support of $\nabla\chi_R$ gives \eqref{eq:caccioppoli}.
\end{proof}

Let $X$ be any smooth vector field on $\R^3$ which agrees with $B\nabla u$ outside $B_{1/16}$.  The arbitrary compactly supported modification near the pole is harmless; changing it only changes the eventual harmonic coefficient by a fixed vector. We have the following endpoint mass bounds for $X$.

\begin{lemma}\label{lem:X-mass}
For $R\ge2$,
\begin{equation}\label{eq:X-annular-mass}
        \int_{\Ann(R)}\abs{X}\,dx\le C\eta(R).
\end{equation}
Consequently, if
$$
        S_j=B_{2^{j+1}}\setminus B_{2^j},\qquad
        m_j=\int_{S_j}\abs{X}\,dx,\qquad j\ge0,
$$
then
\begin{equation}\label{eq:mj-zero}
        m_j\to0.
\end{equation}
Moreover,
\begin{equation}\label{eq:weighted-X}
        \int_{\R^3}\frac{\abs{X(x)}}{1+\abs{x}^2}\,dx<\infty.
\end{equation}
\end{lemma}

\begin{proof}
On $\Ann(R)$, $\abs{B}\le \eta(R)/R$.  Therefore, by Lemma~\ref{lem:caccioppoli},
$$
\begin{aligned}
        \int_{\Ann(R)}\abs{X}\,dx
        &\le \frac{\eta(R)}{R}\int_{\Ann(R)}\abs{\nabla u}\,dx                                      \\
        &\le \frac{\eta(R)}{R}
        \left(\int_{\Ann(R)}\abs{\nabla u}^2\,dx\right)^{1/2}\abs{\Ann(R)}^{1/2}                     \\
        &\le \frac{\eta(R)}{R}\cdot CR^{-1/2}\cdot CR^{3/2}
        \le C\eta(R).
\end{aligned}
$$
Since each $S_j$ is contained in a fixed finite union of annuli of the form $\Ann(2^j)$, \eqref{eq:mj-zero} follows.  Finally,
$$
        \sum_{j\ge0}\int_{S_j}\frac{\abs{X(x)}}{1+\abs{x}^2}\,dx
        \le C\sum_{j\ge0}2^{-2j}m_j<\infty,
$$
because $m_j$ is bounded and tends to zero.  The compact part is finite by construction.
\end{proof}

\section{Endpoint Newtonian potential}
\label{sec:potential}

Let
$$
        K(z)=\frac{z}{\abs z^3},\qquad z\ne0.
$$
Formally, the solution of $\Delta w=\divE X$ is
\begin{equation}\label{eq:w-def}
        w(x)=\frac1{4\pi}\int_{\R^3}\ip{K(x-y)}{X(y)}\,dy.
\end{equation}
At the endpoint $X$ may not belong to $L^1$, so the construction is justified using the weighted estimate \eqref{eq:weighted-X}.

\begin{lemma}\label{lem:weighted-convolution}
The integral in \eqref{eq:w-def} is absolutely convergent for a.e. $x$, defines a function $w\in L^1_{\loc}(\R^3)$, and satisfies
\begin{equation}\label{eq:w-equation}
        \Delta w=\divE X
\end{equation}
in the sense of distributions.
\end{lemma}

\begin{proof}
Since $X\in L^1_{\loc}(\R^3)$, let $B(x_0,r)$ be a compact ball.  Split the $y$-integration into
$$
        E_1=B_{2\abs{x_0}+4r+4},
        \qquad E_2=\R^3\setminus E_1.
$$
For $y\in E_1$,
$$
        \int_{B(x_0,r)}\frac{dx}{\abs{x-y}^2}\le Cr,
$$
because $\abs{x-y}^{-2}$ is locally integrable in dimension three.  Hence
$$
        \int_{B(x_0,r)}\int_{E_1}\frac{\abs{X(y)}}{\abs{x-y}^2}\,dy\,dx
        \le Cr\int_{E_1}\abs{X(y)}\,dy<\infty.
$$
For $y\in E_2$ and $x\in B(x_0,r)$, one has $\abs{x-y}\ge c\abs y$.  Thus
$$
        \int_{B(x_0,r)}\int_{E_2}\frac{\abs{X(y)}}{\abs{x-y}^2}\,dy\,dx
        \le C\abs{B(x_0,r)}\int_{E_2}\frac{\abs{X(y)}}{\abs y^2}\,dy<\infty
$$
by \eqref{eq:weighted-X}.  Tonelli's theorem gives $w\in L^1_{\loc}$ and absolute convergence for a.e. $x$.

For the distributional identity, first truncate
$$
        w_N(x)=\frac1{4\pi}\int_{\abs y\le 2^N}\ip{K(x-y)}{X(y)}\,dy.
$$
The preceding estimates show $w_N\to w$ in $L^1_{\loc}$.  If $\varphi\in C_c^\infty(\R^3)$, then Fubini is justified for $w_N$.  Since
$$
        K(x-y)=-\nabla_x |x-y|^{-1},
        \qquad \Delta_x |x-y|^{-1}=-4\pi\delta_y,
$$
we have
$$
        \int_{\R^3}K(x-y)\Delta\varphi(x)\,dx=-4\pi\nabla\varphi(y).
$$
Thus
$$
\begin{aligned}
        \int_{\R^3}w_N\Delta\varphi\,dx
        &=\frac1{4\pi}\int_{\abs y\le2^N}X(y)\cdot
          \left(\int_{\R^3}K(x-y)\Delta\varphi(x)\,dx\right)dy       \\
        &=-\int_{\abs y\le2^N}X(y)\cdot\nabla\varphi(y)\,dy.
\end{aligned}
$$
Letting $N\to\infty$ gives
$$
        \int w\Delta\varphi=-\int X\cdot\nabla\varphi,
$$
which is exactly \eqref{eq:w-equation}.
\end{proof}

For dyadic $R=2^N$ and for a fixed parameter $\sigma>0$, define 
\begin{equation}\label{eq:P-R-def}
        P_R^\sigma=\frac1{4\pi}\int_{\abs{x}\le \sigma R}X(x)\,dx.
\end{equation}
The parameter $\sigma$ will be chosen once the target annulus is fixed.  

\begin{lemma}\label{lem:PR-growth}
For every fixed $\sigma>0$, along dyadic $R\to\infty$,
\begin{equation}\label{eq:PR-sublog}
        \abs{P_R^\sigma}=o(\log R).
\end{equation}
Moreover, if $0<\sigma_1<\sigma_2<\infty$ are fixed, then
\begin{equation}\label{eq:change-sigma}
        P_R^{\sigma_2}-P_R^{\sigma_1}=o(1).
\end{equation}
\end{lemma}

\begin{proof}
With $m_j$ as in Lemma~\ref{lem:X-mass}, choose an integer $C_\sigma$ such that $\sigma 2^N\le2^{N+C_\sigma}$ for all $N$.  Then
$$
        \abs{P_{2^N}^\sigma}
        \le C+\frac1{4\pi}\sum_{0\le j\le N+C_\sigma}m_j.
$$
Since $m_j\to0$, Lemma \ref{lem:cesaro} gives
$$
        N^{-1}\sum_{0\le j\le N+C_\sigma}m_j\to0.
$$
Hence $\abs{P_{2^N}^\sigma}=o(N)=o(\log R)$.

For \eqref{eq:change-sigma}, the difference is the integral of $X$ over the annulus $\sigma_1R<\abs x\le \sigma_2R$.  This annulus is covered by a number of dyadic shells depending only on $\sigma_1,\sigma_2$, and the mass of each such shell is $o(1)$.
\end{proof}

We obtain an expansion for $w$ depending on $P_R^\sigma$.
\begin{proposition}\label{prop:w-moving}
Let
$$
        \Omega_{\lambda,\Lambda}=B_\Lambda\setminus B_\lambda,
        \qquad 0<\lambda<\Lambda<\infty,
$$
and fix $0<\sigma<\lambda/4$.  For every $1\le q<3/2$, along dyadic $R\to\infty$,
\begin{equation}\label{eq:w-moving}
        \left\|R^2w(R\cdot)-\frac{\ip{P_R^\sigma}{y}}{\abs y^3}\right\|_{L^q(\Omega_{\lambda,\Lambda})}\to0.
\end{equation}
\end{proposition}

\begin{proof}
Fix $y\in\Omega_{\lambda,\Lambda}$ and set $x=Ry$.  Choose a constant $M>2\Lambda$, for instance $M=4\Lambda+4$.  Decompose the $z$-space into
$$
        \Omega_{\mathrm{in}}=B_{\sigma R},\qquad
        \Omega_{\mathrm{mid}}=B_{MR}\setminus B_{\sigma R},\qquad
        \Omega_{\mathrm{out}}=\R^3\setminus B_{MR}.
$$
Then
$$
        R^2w(Ry)-\frac{\ip{P_R^\sigma}{y}}{\abs y^3}
        =I_{\mathrm{in}}(y)+I_{\mathrm{mid}}(y)+I_{\mathrm{out}}(y),
$$
where
$$
\begin{aligned}
        I_{\mathrm{in}}(y)
        &=\frac{R^2}{4\pi}\int_{\abs z\le \sigma R}
        \ip{K(Ry-z)-K(Ry)}{X(z)}\,dz,                                      \\
        I_{\mathrm{mid}}(y)
        &=\frac{R^2}{4\pi}\int_{\sigma R<\abs z<MR}
        \ip{K(Ry-z)}{X(z)}\,dz,                                             \\
        I_{\mathrm{out}}(y)
        &=\frac{R^2}{4\pi}\int_{\abs z\ge MR}
        \ip{K(Ry-z)}{X(z)}\,dz.
\end{aligned}
$$
For the inner part, $\abs z\le \sigma R$ and $\abs{Ry}\ge\lambda R$ imply
$$
        \abs{Ry-tz}\ge(\lambda-\sigma)R\ge \frac34\lambda R,
        \qquad 0\le t\le1.
$$
Since $\abs{\nabla K(\xi)}\le C\abs{\xi}^{-3}$,
$$
        \abs{K(Ry-z)-K(Ry)}\le C_{\lambda,\sigma}\frac{\abs z}{R^3}.
$$
Therefore
$$
        \abs{I_{\mathrm{in}}(y)}
        \le \frac{C}{R}\int_{\abs z\le \sigma R}\abs z\abs{X(z)}\,dz
        \le \frac{C}{R}\sum_{0\le j\le N+C_\sigma}2^j m_j=o(1),
$$
by Lemma \ref{lem:cesaro}.  This convergence is uniform in $y$.

For the middle part, Minkowski's inequality gives
$$
\begin{aligned}
        \norm{I_{\mathrm{mid}}}_{L^q(\Omega_{\lambda,\Lambda})}
        &\le C\int_{\sigma R<\abs z<MR}\abs{X(z)}
        \left(\int_{\Omega_{\lambda,\Lambda}}
        \frac{R^{2q}}{\abs{Ry-z}^{2q}}\,dy\right)^{1/q}dz.
\end{aligned}
$$
Writing $Z=z/R$,
$$
        \int_{\Omega_{\lambda,\Lambda}}
        \frac{R^{2q}}{\abs{Ry-z}^{2q}}\,dy
        =\int_{\Omega_{\lambda,\Lambda}}\frac{dy}{\abs{y-Z}^{2q}}
        \le C_{q,\lambda,\Lambda,M},
$$
because $2q<3$.  Hence
$$
        \norm{I_{\mathrm{mid}}}_{L^q}
        \le C\int_{\sigma R<\abs z<MR}\abs{X(z)}\,dz=o(1),
$$
since the latter region consists of only finitely many dyadic shells around scale $R$, with the number depending only on $\sigma$ and $M$.

For the outer part, $\abs z\ge MR$ and $\abs{Ry}\le\Lambda R$ imply $\abs{Ry-z}\ge c\abs z$.  Choose an integer $C_M$ depending only on $M$ such that $2^{C_M}\le M$.  Then
$$
        \abs{I_{\mathrm{out}}(y)}
        \le C R^2\sum_{\ell\ge0}
        \int_{2^{N+C_M+\ell}\le\abs z\le2^{N+C_M+\ell+1}}
        \frac{\abs{X(z)}}{\abs z^2}\,dz
        \le C\sum_{\ell\ge0}4^{-\ell}m_{N+C_M+\ell}=o(1),
$$
again by Lemma \ref{lem:cesaro}.  This estimate is uniform in $y$.  Combining the three bounds proves \eqref{eq:w-moving}.
\end{proof}

\section{The harmonic remainder and the endpoint Green expansion}
\label{sec:green}

Set
\begin{equation}\label{eq:h-def}
        h=u-w.
\end{equation}
Outside a compact set containing the pole, one has $X=B\nabla u$ and
$\Delta u=\divE(B\nabla u)$.  Since Lemma~\ref{lem:weighted-convolution} gives
$\Delta w=\divE X$, we obtain $\Delta h=0$ there in the distributional sense.
Hence $h$ is smooth and harmonic there.

\begin{lemma}[Averaged decay of $w$]\label{lem:w-average-decay}
For the annuli $\Ann(R)=B_{8R}\setminus B_{R/8}$,
\begin{equation}\label{eq:w-average-decay}
        \fint_{\Ann(R)}\abs{w(x)}\,dx\to0.
\end{equation}
\end{lemma}

\begin{proof}
Let $R=2^N$ first.  The proof is again a three-region estimate.  For $x\in\Ann(R)$, decompose the integral defining $w$ into $\abs z\le R/16$, $R/16<\abs z<16R$, and $\abs z\ge16R$.

For the inner region, write
\[
\begin{aligned}
        \frac1{4\pi}\int_{\abs z\le R/16}\ip{K(x-z)}{X(z)}\,dz
        &=\ip{K(x)}{P_R^{1/16}}+\mathcal E_{\mathrm{in}}(x),
\end{aligned}
\]
where
$$
        \mathcal E_{\mathrm{in}}(x):=\frac1{4\pi}\int_{\abs z\le R/16}
        \ip{K(x-z)-K(x)}{X(z)}\,dz .
$$

The main term is bounded by $C\abs{P_R^{1/16}}R^{-2}=o(1)$.  The error satisfies
$$
        \abs{\mathcal E_{\mathrm{in}}(x)}
        \le CR^{-3}\int_{\abs z\le R/16}\abs z\abs X(z)\,dz=o(R^{-2}),
$$
by Lemma~\ref{lem:cesaro}.

For the middle region, Tonelli's theorem and the local integrability of $\abs{x-z}^{-2}$ give, for each fixed $z$,
$$
        \int_{\Ann(R)}\frac{dx}{\abs{x-z}^2}\le CR.
$$
Therefore the averaged contribution is at most
$$
        C R^{-3}\cdot R\int_{R/16<\abs z<16R}\abs X(z)\,dz=o(1).
$$
For the outer region, $\abs{x-z}\ge c\abs z$ and hence the contribution is bounded by
$$
        C\sum_{\ell\ge0}4^{-\ell}m_{N+C+\ell}=o(1).
$$
This proves \eqref{eq:w-average-decay} along dyadic $R$.

For a general $R\to\infty$, choose dyadic $R_j\le R<2R_j$.  Then $\Ann(R)$ is contained in a fixed finite union of annuli comparable to $\Ann(R_j)$ and $\Ann(2R_j)$, while the volume ratios are uniformly bounded.  The dyadic result gives the general result.
\end{proof}

The following exterior harmonic expansion holds.
\begin{lemma}\label{lem:h-expansion}
The harmonic remainder satisfies
\begin{equation}\label{eq:h-to-zero}
        h(x)\to0\qquad\text{as }\abs{x}\to\infty.
\end{equation}
Consequently there exist $c\in\R$ and $b\in\R^3$ such that
\begin{equation}\label{eq:h-expansion}
        h(x)=\frac{c}{\abs x}+\frac{\ip{b}{x}}{\abs x^3}+O(\abs x^{-3})
\end{equation}
as $\abs x\to\infty$, with the corresponding derivative estimates.
\end{lemma}

\begin{proof}
By \eqref{eq:green-two-sided}, $\fint_{\Ann(R)}\abs u\,dx=O(R^{-1})\to0$.  By Lemma~\ref{lem:w-average-decay}, $\fint_{\Ann(R)}\abs w\,dx\to0$.  Hence
$$
        \fint_{\Ann(R)}\abs h\,dx\to0.
$$
If $\abs{x}=2R$, then $B(x,R)\subset \Ann(R)$.  Since $h$ is harmonic in $B(x,R)$ for $R$ sufficiently large and $\abs h$ is subharmonic,
$$
        \abs{h(x)}\le C\fint_{\Ann(R)}\abs h\,dx\to0.
$$
This proves \eqref{eq:h-to-zero}.  The standard spherical-harmonic, or Laurent, expansion of an exterior harmonic function decaying at infinity, see for instance \cite[Chapter~5]{ABR}, then gives \eqref{eq:h-expansion}.
\end{proof}

We write
$$
        \rhoE(y)=\abs y^{-1}
$$
on $\R^3\setminus\{0\}$.  For $R>0$, define the rescaled Green function and coefficients
$$
        U_R(y)=Ru(Ry),\qquad A_R(y)=A(Ry),\qquad \mathcal B_R(y)=B(Ry).
$$
Then $U_R$ solves
\begin{equation}\label{eq:UR-equation}
        \divE(A_R\nabla U_R)=0
\end{equation}
on every fixed annulus in $\R^3\setminus\{0\}$.

\begin{lemma}\label{lem:UR-Lq}
For every $0<\lambda<\Lambda<\infty$ and every $1\le q<3/2$,
\begin{equation}\label{eq:UR-Lq}
        U_R\to c\rhoE \qquad \text{in} \ L^q(B_\Lambda\setminus B_\lambda)
\end{equation}
 along dyadic $R\to\infty$.
\end{lemma}

\begin{proof}
By $u=w+h$ and \eqref{eq:h-expansion},
$$
        U_R(y)=Rw(Ry)+\frac{c}{\abs y}+O(R^{-1})
$$
on $B_\Lambda\setminus B_\lambda$.  Choose any fixed $0<\sigma<\lambda/4$.  Proposition~\ref{prop:w-moving} and Lemma~\ref{lem:PR-growth} imply
$$
\begin{aligned}
        \norm{Rw(R\cdot)}_{L^q}
        &\le R^{-1}\left\|R^2w(R\cdot)-\frac{\ip{P_R^\sigma}{y}}{\abs y^3}\right\|_{L^q}
        +CR^{-1}\abs{P_R^\sigma}  \\
        &=o(1).
\end{aligned}
$$
Thus \eqref{eq:UR-Lq} follows.
\end{proof}

\begin{proposition}\label{prop:first-W1p}
For every finite $p>3$ and every $0<\lambda<\Lambda<\infty$,
\begin{equation}\label{eq:first-W1p}
        U_R\to c\rhoE
        \qquad\text{in }W^{1,p}(B_\Lambda\setminus B_\lambda)
\end{equation}
along dyadic $R\to\infty$.
\end{proposition}

\begin{proof}
Choose $\lambda'<\lambda<\Lambda<\Lambda'$ and set
$$
        \Omega'=B_\Lambda\setminus B_\lambda,
        \qquad \Omega=B_{\Lambda'}\setminus B_{\lambda'}.
$$
Let $e_R=U_R-c\rhoE$.  Since $\Delta\rhoE=0$ on $\Omega$ and $A_R=I-\mathcal B_R$,
$$
        \divE(A_R\nabla e_R)=c\,\divE(\mathcal B_R\nabla\rhoE).
$$
On $\Omega$,
\begin{equation}\label{eq:B-rho-small}
        \norm{\mathcal B_R\nabla\rhoE}_{L^\infty(\Omega)}=o(R^{-1})=o(1).
\end{equation}
By Lemma~\ref{lem:UR-Lq}, $e_R\to0$ in $L^q(\Omega)$ for every $1\le q<3/2$.  The right-hand side field $Y_R=c\mathcal B_R\nabla\rhoE$ satisfies $\norm{Y_R}_{L^\infty(\Omega)}\to0$ by \eqref{eq:B-rho-small}, and $A_R\to I$ uniformly on $\Omega$.  The sequence upgrade Proposition~\ref{prop:elliptic-sequence} therefore gives $e_R\to0$ in $W^{1,p}(\Omega')$, which is exactly \eqref{eq:first-W1p}.
\end{proof}

We next identify the coefficient of the leading \(r^{-1}\) term.
\begin{proposition}\label{prop:c=1}
With the normalization above, we have
\begin{equation}\label{eq:c-equals-one}
        c=1.
\end{equation}
\end{proposition}

\begin{proof}
First $c>0$.  Indeed, by the lower Green estimate in \eqref{eq:green-two-sided}, $U_R(y)\ge c_0\abs y^{-1}$ on each fixed annulus, while Proposition~\ref{prop:first-W1p} gives $U_R\to c\abs y^{-1}$ locally uniformly away from the origin.

We now localize the relevant level sets.  Choose a non-negative function
$$
        \zeta\in C_c^\infty((c/3,c/2)),\qquad \int_{c/3}^{c/2}\zeta(t)\,dt=1.
$$
Since $U_R\to c\abs y^{-1}$ uniformly on neighborhoods of the spheres $\abs y=2$ and $\abs y=3$, for all sufficiently large dyadic $R$ and all $t\in\spt\zeta$,
\begin{equation}\label{eq:boundary-level-signs}
        U_R>t \text{ on }\{\abs y=2\},\qquad
        U_R<t \text{ on }\{\abs y=3\}.
\end{equation}
We claim that, for such $R$ and $t$,
\begin{equation}\label{eq:level-localization}
        \{U_R=t\}\subset B_3\setminus \overline{B_2}.
\end{equation}
To see the inner exclusion, fix $t\in\spt\zeta$.  Since $u$ has a positive Green pole at the origin, $U_R(y)\to\infty$ as $y\to0$.  Hence $U_R>t$ on $\{\abs y=\varepsilon\}$ for $\varepsilon$ small.  Applying the maximum principle to $B_2\setminus \overline{B_\varepsilon}$, using \eqref{eq:boundary-level-signs} and then letting $\varepsilon\downarrow0$, gives $U_R>t$ on $B_2\setminus\{0\}$.  Thus $\{U_R=t\}\cap B_2=\varnothing$.

For the outer exclusion, choose $L_0>3$ so large that $C_0/L_0<\inf\spt\zeta$, where $C_0$ comes from \eqref{eq:green-two-sided}.  The uniform upper Green bound $U_R(y)\le C_0\abs y^{-1}$ gives $U_R<t$ on $\{\abs y=L_0\}$, and also on $\{\abs y\ge L_0\}$, for all large $R$ and all $t\in\spt\zeta$.  The maximum principle on $B_{L_0}\setminus\overline{B_3}$, together with \eqref{eq:boundary-level-signs}, gives $U_R<t$ on $B_{L_0}\setminus B_3$.  Therefore $U_R<t$ on $\R^3\setminus B_3$.  This proves \eqref{eq:level-localization}.

For a.e. $t\in\spt\zeta$, the level $\{U_R=t\}$ is regular.  By \eqref{eq:level-localization}, the whole level set lies in $B_3\setminus B_2$, and the normalization gives
$$
        \int_{\{U_R=t\}}\abs{\nabla_{g_R}U_R}\,dA_{g_R}=4\pi,
        \qquad g_R(y)=g(Ry).
$$
Here $g_R$ denotes the coefficient rescaling of the metric matrix, not the pullback metric under $y\mapsto Ry$.  Therefore the coarea formula on the fixed annulus gives
\begin{equation}\label{eq:c-coarea-normalization}
\begin{aligned}
        \int_{B_3\setminus B_2}
        \zeta(U_R)\abs{\nabla_{g_R}U_R}^2\,dv_{g_R}
        &=\int_{\spt\zeta}
        \zeta(t)
        \left(\int_{\{U_R=t\}}\abs{\nabla_{g_R}U_R}\,dA_{g_R}\right)dt  \\
        &=4\pi.
\end{aligned}
\end{equation}

Now $U_R\to c\abs y^{-1}$ in $W^{1,p}(B_3\setminus B_2)$, with $p>3$, and $g_R\to \delta$ in $C^0$ on this annulus.  Passing to the limit in \eqref{eq:c-coarea-normalization} yields
$$
        4\pi
        =\int_{B_3\setminus B_2}
        \zeta(c/\abs y)\frac{c^2}{\abs y^4}\,dy
        =4\pi\int_2^3 \zeta(c/r)c^2 r^{-2}\,dr.
$$
With $t=c/r$, this last integral is
$$
        4\pi c\int_{c/3}^{c/2}\zeta(t)\,dt=4\pi c.
$$
Thus $c=1$.
\end{proof}

Finally, we prove endpoint Green expansion.
\begin{proposition}\label{prop:endpoint-green}
Let $p>3$ be finite and let $0<\lambda<\Lambda<\infty$.  Fix $0<\sigma<\lambda/8$.  Along dyadic $R\to\infty$, define
\begin{equation}\label{eq:Q-R-def}
        Q_R^\sigma=b+P_R^\sigma.
\end{equation}
Although $b$ and $P_R^\sigma$ separately depend on the compactly supported modification of $X$ near the pole, they are understood relative to the fixed choice of that modification.  Then
\begin{equation}\label{eq:Q-growth}
        \abs{Q_R^\sigma}=o(\log R)
\end{equation}
and
\begin{equation}\label{eq:endpoint-green}
        \left\|R\left(U_R-\rhoE\right)-\frac{\ip{Q_R^\sigma}{y}}{\abs y^3}\right\|_{W^{1,p}(B_\Lambda\setminus B_\lambda)}\to0.
\end{equation}
\end{proposition}

\begin{proof}
The growth estimate follows from Lemma~\ref{lem:PR-growth}.  Let
$$
        d_R(y)=\frac{\ip{Q_R^\sigma}{y}}{\abs y^3},
        \qquad \mathrm{Err}_R(y)=R(U_R-\rhoE)-d_R(y).
$$
Set $\lambda'=\lambda/2$ and choose $\Lambda'>\Lambda$.  Since $\sigma<\lambda/8=\lambda'/4$, Proposition~\ref{prop:w-moving} applies on $B_{\Lambda'}\setminus B_{\lambda'}$.  Combining it with Lemma~\ref{lem:h-expansion} and Proposition~\ref{prop:c=1}, we get
\begin{equation}\label{eq:E-Lq-zero}
        \mathrm{Err}_R\to0\qquad\text{in }L^q(B_{\Lambda'}\setminus B_{\lambda'})
\end{equation}
for every $1\le q<3/2$.

Since $\Delta\rhoE=\Delta d_R=0$ away from the origin and $A_R=I-\mathcal B_R$, the equation \eqref{eq:UR-equation} gives
$$
\begin{aligned}
        \divE(A_R\nabla \mathrm{Err}_R)
        &=-R\divE(A_R\nabla\rhoE)-\divE(A_R\nabla d_R)      \\
        &=\divE\left(R\mathcal B_R\nabla\rhoE+\mathcal B_R\nabla d_R\right).
\end{aligned}
$$
On $B_{\Lambda'}\setminus B_{\lambda'}$,
\begin{equation}\label{eq:Y-R-small}
        \norm{R\mathcal B_R\nabla\rhoE+\mathcal B_R\nabla d_R}_{L^\infty}
        \le o(1)+C\abs{Q_R^\sigma}\,O(R^{-1})=o(1),
\end{equation}
because $\abs{Q_R^\sigma}=o(\log R)$.

Apply the sequence upgrade Proposition~\ref{prop:elliptic-sequence} with
$$
        v_i=\mathrm{Err}_R,
        \qquad
        Y_i=R\mathcal B_R\nabla\rhoE+\mathcal B_R\nabla d_R,
        \qquad
        \Omega=B_{\Lambda'}\setminus B_{\lambda'},
        \qquad
        \Omega'=B_\Lambda\setminus B_\lambda.
$$
The hypotheses are exactly \eqref{eq:E-Lq-zero}, \eqref{eq:Y-R-small}, and $A_R\to I$ in $L^\infty$ on the fixed annulus.  Therefore $\mathrm{Err}_R\to0$ in $W^{1,p}(B_\Lambda\setminus B_\lambda)$, which proves \eqref{eq:endpoint-green}.
\end{proof}

\section{\texorpdfstring{The $F$- and $D$-functionals}{The F- and D-functionals}}
\label{sec:D}

Assume now that $R_g\ge0$.  For $t>0$ such that $1/t$ is a regular value of $u$, define
\begin{equation}\label{eq:F-def}
        F(t)=4\pi t-t^2\int_{\{u=1/t\}}H_g\abs{\nabla_g u}\,dA_g
        +t^3\int_{\{u=1/t\}}\abs{\nabla_g u}^2\,dA_g.
\end{equation}
Here
$$
        H_g=-\operatorname{div}_g\left(\frac{\nabla_g u}{|\nabla_g u|}\right)
$$
is the mean curvature of the level set with respect to the normal $\nabla_g u/|\nabla_g u|$, in the sign convention used throughout the paper, see \cite{AMO} or Appendix \ref{app:FD-endpoint}. By Proposition~\ref{prop:F-functional-facts}, the right-continuous representative of $F$ satisfies
\begin{equation}\label{eq:F-nonneg-monotone}
        F(t)\ge0,
        \qquad F\text{ is non-decreasing.}
\end{equation}

Fix $\psi\in C_c^\infty((0,1))$ with $\psi\ge0$ and $\int_0^1\psi=1$.  Define
\begin{equation}\label{eq:E-def}
        \Efun(a,s)=\int_{a+s}^{2a+2s}\frac{F(t)}{t^3}\,dt,
        \qquad 0\le s\le a,
\end{equation}
and
\begin{equation}\label{eq:D-def}
        D(a)=\int_0^a\psi(s/a)\Efun(a,s)\,ds.
\end{equation}
As in Mazurowski--Yao \cite[Propositions~14 and~15]{MYrigidity},
\begin{equation}\label{eq:D-positive-monotone}
        D(a)\ge0,
        \qquad a\mapsto aD(a)\text{ is non-decreasing}.
\end{equation}
Indeed,
$$
        aD(a)=\int_0^1\frac{\psi(\vartheta)}{(1+\vartheta)^2}
        \left(\int_1^2\frac{F((1+\vartheta)a\tau)}{\tau^3}\,d\tau\right)d\vartheta,
$$
which is monotone in $a$ by \eqref{eq:F-nonneg-monotone}.

Using the coarea formula and the identity for the derivative of $\int_{\{u=1/t\}}\abs{\nabla_g u}^2dA_g$ recorded in this normalization in \cite[Appendix~A, Proposition~25]{MYrigidity}, one obtains
\begin{equation}\label{eq:D-unscaled-integral}
        D(a)=c_\psi+\int_{\R^3}\phi(au(x))\abs{\nabla_g u(x)}^3a^3\,dv_g(x),
\end{equation}
where
\begin{equation}\label{eq:c-psi}
        c_\psi=2\pi\int_0^1\frac{\psi(s)}{1+s}\,ds
\end{equation}
and
\begin{equation}\label{eq:phi-def}
        \phi(t)=\frac1{t^3}\left[\frac12\psi\left(\frac1{2t}-1\right)-\psi\left(\frac1t-1\right)\right].
\end{equation}
Since $\psi\in C_c^\infty((0,1))$,
\begin{equation}\label{eq:phi-support}
        \spt\phi\Subset(1/4,1).
\end{equation}
For $a>0$, set
$$
        u_a(y)=a u(ay),
        \qquad g_a(y)=g(ay),
        \qquad \rhoE(y)=\abs y^{-1}.
$$
Here $g_a$ denotes the coefficient rescaling of the metric matrix, not the pullback metric under $y\mapsto ay$.  With this convention,
$$
        |\nabla_{g_a}u_a(y)|=a^2|\nabla_g u(ay)|,
        \qquad
        dv_{g_a}(y)=\sqrt{\det(g_{kl})(ay)}\,dy.
$$
After the change of variables $x=ay$,
\begin{equation}\label{eq:D-rescaled-allspace}
        D(a)=c_\psi+\int_{\R^3}\phi(u_a(y))\abs{\nabla_{g_a}u_a(y)}^3\,dv_{g_a}(y).
\end{equation}

\begin{lemma}\label{lem:fixed-annulus}
There is a fixed annulus
$$
        \Omega=B_\Lambda\setminus B_\lambda,
        \qquad 0<\lambda<\Lambda<\infty,
$$
such that for all sufficiently large $a$,
\begin{equation}\label{eq:D-fixed-annulus}
        D(a)=c_\psi+\int_{\Omega}\phi(u_a)\abs{\nabla_{g_a}u_a}^3\,dv_{g_a}.
\end{equation}
Moreover, $\Omega$ may be chosen so that
\begin{equation}\label{eq:rho-support-in-omega}
        \spt\bigl(\phi(\rhoE)\bigr)\Subset\Omega.
\end{equation}
\end{lemma}

\begin{proof}
Choose $0<\alpha<\beta<\infty$ with $\spt\phi\subset[\alpha,\beta]$. By \eqref{eq:phi-support}, $1/4<\alpha<\beta<1$. Choose $\lambda,\Lambda$ so that
\begin{equation}\label{eq:lambda-Lambda-choice}
        0<\lambda<\min\left\{\frac{c_0}{2\beta},\frac1{2\beta}\right\},
        \qquad
        \Lambda>\max\left\{\frac{2C_0}{\alpha},\frac2{\alpha}\right\}.
\end{equation}
The two-sided Green estimate \eqref{eq:green-two-sided} implies that, for $\abs{ay}\ge1$,
$$
        \frac{c_0}{\abs y}\le u_a(y)\le \frac{C_0}{\abs y}.
$$
For $\abs y\ge\Lambda$ and $a$ large, the upper bound gives $u_a(y)<\alpha$, hence $\phi(u_a(y))=0$. For $\abs y\le\lambda$ and $\abs{ay}\ge1$, the lower bound gives $u_a(y)>\beta$, hence $\phi(u_a(y))=0$. It remains only to consider the compact region $\abs{ay}\le1$ with $\abs y\le\lambda$. This means $\abs y\le a^{-1}$. Near the Green pole, $u(x)\to\infty$ as $x\to0$, and $u$ is positive and continuous away from the pole. Hence
$$
        m_0:=\inf_{0<\abs x\le1}u(x)>0.
$$
On the remaining region $\abs{ay}\le1$, one has $a u(ay)\ge a m_0>\beta$ for all sufficiently large $a$. Thus the integrand in \eqref{eq:D-rescaled-allspace} vanishes outside $\Omega$ for all large $a$. Finally, if $\phi(\rhoE(y))\ne0$, then $\alpha\le |y|^{-1}\le\beta$, hence $1/\beta\le |y|\le 1/\alpha$. The extra inequalities in \eqref{eq:lambda-Lambda-choice} imply $\lambda<1/\beta$ and $\Lambda>1/\alpha$, which proves \eqref{eq:rho-support-in-omega}.
\end{proof}

On this fixed annulus define the Banach-space functional
\begin{equation}\label{eq:Dfun-def}
        \Dfun(\gamma,f)=c_\psi+\int_\Omega\phi(f)\abs{\nabla_\gamma f}^3\,dv_\gamma,
\end{equation}
for continuous Riemannian metrics $\gamma$ close to $\delta$ and functions $f\in W^{1,p}(\Omega)$ close to $\rhoE$, with $p>3$.  For large dyadic $a$,
$$
        D(a)=\Dfun(g_a,u_a).
$$

The following lemma can be found in \cite[Proposition 16]{MYrigidity}.
\begin{lemma}\label{lem:euclidean-calibration}
One has
\begin{equation}\label{eq:D-euclidean-zero}
        \Dfun(\delta,\rhoE)=0.
\end{equation}
\end{lemma}

\begin{proof}
By Lemma~\ref{lem:fixed-annulus}, $\spt(\phi(\rhoE))\Subset\Omega$.  Since $\rhoE=r^{-1}$ and $\abs{\nabla\rhoE}=r^{-2}$,
$$
        \int_\Omega \phi(\rhoE)\abs{\nabla\rhoE}^3\,dy
        =4\pi\int_0^\infty \phi(r^{-1})r^{-4}\,dr.
$$
With $\tau=r^{-1}$, this becomes
$$
        4\pi\int_0^\infty \phi(\tau)\tau^2\,d\tau.
$$
Using the definition
$$
        \phi(\tau)=\tau^{-3}\left[\frac12\psi\left(\frac1{2\tau}-1\right)-\psi\left(\frac1{\tau}-1\right)\right],
$$
we compute
$$
\begin{aligned}
        \int_0^\infty \phi(\tau)\tau^2\,d\tau
        &=\frac12\int_0^\infty \tau^{-1}\psi\left(\frac1{2\tau}-1\right)\,d\tau
          -\int_0^\infty \tau^{-1}\psi\left(\frac1{\tau}-1\right)\,d\tau       \\
        &=\frac12\int_0^1 \frac{\psi(s)}{1+s}\,ds
          -\int_0^1 \frac{\psi(s)}{1+s}\,ds                                  \\
        &=-\frac12\int_0^1 \frac{\psi(s)}{1+s}\,ds.
\end{aligned}
$$
Therefore
$$
        \int_\Omega \phi(\rhoE)\abs{\nabla\rhoE}^3\,dy
        =-2\pi\int_0^1\frac{\psi(s)}{1+s}\,ds=-c_\psi.
$$
The definition of $\Dfun$ gives $\Dfun(\delta,\rhoE)=c_\psi-c_\psi=0$.
\end{proof}

\section{\texorpdfstring{Quadratic expansion of the $D$-functional}{Quadratic expansion of the D-functional}}
\label{sec:D-quadratic}

We prove the quadratic expansion at the Euclidean model.
\begin{proposition}\label{prop:D-quadratic}
Fix $p>3$.  There exist $\delta_0>0$ and $C<\infty$, depending on $\Omega$, $p$, and $\psi$, such that if $k$ is a continuous symmetric two-tensor on $\Omega$ and $v\in W^{1,p}(\Omega)$ satisfy
$$
        \norm{k}_{C^0(\Omega)}+
        \norm{v}_{W^{1,p}(\Omega)}\le\delta_0,
$$
then
\begin{equation}\label{eq:D-quadratic}
        \Dfun(\delta+k,\rhoE+v)
        =\Dfun(\delta,\rhoE)+\LD(k,v)+\mathcal R(k,v),
\end{equation}
where
\begin{align}\label{eq:L-def}
        \LD(k,v)=
        &\int_\Omega\phi'(\rhoE)\abs{\nabla\rhoE}^3v\,dy
        +3\int_\Omega\phi(\rhoE)\abs{\nabla\rhoE}\ip{\nabla\rhoE}{\nabla v}\,dy\notag\\
        &+\int_\Omega\phi(\rhoE)
        \left[\frac12(\tr k)\abs{\nabla\rhoE}^3
        -\frac32\abs{\nabla\rhoE}\,k(\nabla\rhoE,\nabla\rhoE)\right]dy,
\end{align}
and
\begin{equation}\label{eq:D-rem-bound}
        \abs{\mathcal R(k,v)}
        \le C\left(\norm{k}_{C^0(\Omega)}+
        \norm{v}_{W^{1,p}(\Omega)}\right)^2.
\end{equation}
\end{proposition}

\begin{proof}
Set
$$
        \alpha=\norm{k}_{C^0(\Omega)},\qquad
        \beta=\norm{v}_{W^{1,p}(\Omega)},\qquad
        \theta=\alpha+\beta.
$$
We may assume $\theta\le1$ by decreasing $\delta_0$.  Since $p>3$,
$W^{1,p}(\Omega)\hookrightarrow C^0(\Omega)$, so
\begin{equation}\label{eq:Sobolev-small-v4}
        \norm{v}_{C^0}+\norm{\nabla v}_{L^r}
        \le C_r\beta,
        \qquad 1\le r\le p.
\end{equation}
The annulus $\Omega$ is separated from the origin, hence $\rhoE$ and $\nabla\rhoE$ are smooth and bounded on $\Omega$.  Decreasing $\delta_0$ if necessary, $\delta+k$ is uniformly positive definite and $\rhoE+v$ stays in a compact subset of $(0,\infty)$ on which all derivatives of $\phi$ are bounded.

First keep the metric Euclidean.  Taylor's theorem gives the pointwise estimate
\begin{equation}\label{eq:phi-remainder-v4}
        \phi(\rhoE+v)=\phi(\rhoE)+\phi'(\rhoE)v+R_\phi,
        \qquad |R_\phi|\le C\norm{v}_{C^0}^2.
\end{equation}
The elementary cubic inequality
\begin{equation}\label{eq:cubic-estimate-v4}
        \left||X+Y|^3-|X|^3-3|X|\ip{X}{Y}\right|
        \le C(|X||Y|^2+|Y|^3)
\end{equation}
with $X=\nabla\rhoE$ and $Y=\nabla v$ gives
\begin{equation}\label{eq:cubic-rho-v4}
        |\nabla(\rhoE+v)|^3
        =|\nabla\rhoE|^3
        +3|\nabla\rhoE|\ip{\nabla\rhoE}{\nabla v}
        +R_\nabla,
        \qquad
        |R_\nabla|\le C(|\nabla v|^2+|\nabla v|^3).
\end{equation}
Multiplying \eqref{eq:phi-remainder-v4} and \eqref{eq:cubic-rho-v4}, the terms linear in $v$ are exactly
$$
        \int_\Omega \phi'(\rhoE)|\nabla\rhoE|^3v\,dy
        +3\int_\Omega \phi(\rhoE)|\nabla\rhoE|
        \ip{\nabla\rhoE}{\nabla v}\,dy.
$$
All remaining function-only terms are quadratic.  Indeed, using \eqref{eq:Sobolev-small-v4},
$$
        \int_\Omega |R_\phi|\,dy\le C\beta^2,
        \qquad
        \int_\Omega |R_\nabla|\,dy
        \le C(\norm{\nabla v}_{L^2}^2+\norm{\nabla v}_{L^3}^3)
        \le C\beta^2,
$$
and the mixed first-order product obeys
$$
        \int_\Omega |v|\,|\nabla v|\,dy
        \le \norm{v}_{C^0}\norm{\nabla v}_{L^1}
        \le C\beta^2.
$$
The remaining products are also quadratic:
$$
        \int_\Omega |R_\phi|(1+|\nabla v|+|\nabla v|^2+|\nabla v|^3)\,dy
        +\int_\Omega |v|\,|R_\nabla|\,dy
        \le C\beta^2.
$$
Thus
\begin{equation}\label{eq:function-only-quadratic-v4}
        \left|\Dfun(\delta,\rhoE+v)-\Dfun(\delta,\rhoE)-\LD(0,v)\right|
        \le C\beta^2.
\end{equation}

Now include the metric perturbation.  For a positive definite matrix $H$ near $I$ and $\xi\in\R^3$, set
$$
        \Psi(H,\xi)=\ip{H^{-1}\xi}{\xi}^{3/2}\sqrt{\det H}.
$$
Uniform Taylor expansion at $I$ gives
\begin{equation}\label{eq:Psi-Taylor-v4}
        \Psi(I+S,\xi)=|\xi|^3+\mathcal M(S,\xi)+R_{\mathcal M}(S,\xi),
\end{equation}
where
$$
        \mathcal M(S,\xi)=\frac12(\tr S)|\xi|^3
        -\frac32|\xi|S(\xi,\xi),
        \qquad
        |R_{\mathcal M}(S,\xi)|\le C|S|^2|\xi|^3.
$$
Applying this with $S=k$ and $\xi=\nabla(\rhoE+v)$, the pure metric remainder satisfies
$$
        \int_\Omega |\phi(\rhoE+v)R_{\mathcal M}(k,\nabla(\rhoE+v))|\,dy
        \le C\alpha^2\int_\Omega (1+|\nabla v|^3)\,dy
        \le C\alpha^2.
$$
It remains to replace the linear metric term
$\phi(\rhoE+v)\mathcal M(k,\nabla(\rhoE+v))$ by
$\phi(\rhoE)\mathcal M(k,\nabla\rhoE)$.  Since $\nabla\rhoE$ is bounded on $\Omega$, and since $\mathcal M$ is linear in $k$ and cubic in $\xi$,
\begin{equation}\label{eq:M-Lipschitz-v4}
        |\mathcal M(k,\nabla(\rhoE+v))-\mathcal M(k,\nabla\rhoE)|
        \le C\alpha(|\nabla v|+|\nabla v|^2+|\nabla v|^3).
\end{equation}
Moreover,
$$
        |\phi(\rhoE+v)-\phi(\rhoE)|\le C|v|.
$$
Consequently
$$
\begin{aligned}
        &\int_\Omega \left|
        \phi(\rhoE+v)\mathcal M(k,\nabla(\rhoE+v))
        -\phi(\rhoE)\mathcal M(k,\nabla\rhoE)
        \right|dy                                                \\
        &\qquad \le C\alpha\int_\Omega
        (|v|+|\nabla v|+|\nabla v|^2+|\nabla v|^3)dy
        \le C\alpha\beta\le C\theta^2.
\end{aligned}
$$
Combining this estimate with \eqref{eq:function-only-quadratic-v4} and the pure metric remainder, the first variation is precisely the expression in \eqref{eq:L-def}, and every leftover term is bounded by $C\theta^2$.  This proves \eqref{eq:D-rem-bound}.
\end{proof}

The following lemma shows that the term involving in $P_R^\sigma$ in expansion  of Green function $w$ is killed by the first variation.
\begin{lemma}\label{lem:dipole-kernel}
For every $\mathbf q\in\R^3$,
\begin{equation}\label{eq:dipole-kernel}
        \LD\left(0,\frac{\ip{\mathbf q}{y}}{\abs y^3}\right)=0.
\end{equation}
\end{lemma}

\begin{proof}
Let
$$
        v_{\mathbf q}(y)=\frac{\ip{\mathbf q}{y}}{\abs y^3}.
$$
The $v$-part of $\LD$ is
$$
        \LD(0,v_{\mathbf q})=\int_\Omega\phi'(\rhoE)\abs{\nabla\rhoE}^3v_{\mathbf q}\,dy
        +3\int_\Omega\phi(\rhoE)\abs{\nabla\rhoE}\ip{\nabla\rhoE}{\nabla v_{\mathbf q}}\,dy.
$$
Here we use that $\Omega=B_\Lambda\setminus B_\lambda$ is a centered Euclidean annulus.  In polar coordinates $y=r\theta$, one has $\rhoE=r^{-1}$ and $v_{\mathbf q}=r^{-2}\ip{\mathbf q}{\theta}$.  Also
$$
        \partial_r v_{\mathbf q}=-2r^{-3}\ip{\mathbf q}{\theta},
        \qquad \nabla\rhoE=-r^{-2}\theta,
        \qquad \ip{\nabla\rhoE}{\nabla v_{\mathbf q}}=2r^{-5}\ip{\mathbf q}{\theta}.
$$
Thus both integrands are radial functions times $\ip{\mathbf q}{\theta}$.  Since
$$
        \int_{\Sph^2}\ip{\mathbf q}{\theta}\,d\theta=0,
$$
each integral vanishes.
\end{proof}

\section{Vanishing of the endpoint mass quantity}
\label{sec:aD}

\begin{proposition}\label{prop:aD-zero}
Along dyadic $a\to\infty$,
\begin{equation}\label{eq:aD-zero}
        aD(a)\to0.
\end{equation}
\end{proposition}

\begin{proof}
Fix a finite $p>3$.  Let $a=2^N$ and use the fixed annulus $\Omega=B_\Lambda\setminus B_\lambda$ from Lemma~\ref{lem:fixed-annulus}.  Fix once and for all a number $0<\sigma<\lambda/8$.  Set
$$
        k_a=g_a-\delta,
        \qquad v_a=u_a-\rhoE.
$$
Since $\Omega$ is fixed, the endpoint metric decay gives
\begin{equation}\label{eq:ak-zero}
        a\norm{k_a}_{C^0(\Omega)}
        \le \lambda^{-1}\sup_{\lambda a\le |x|\le \Lambda a}|x|\,|g(x)-\delta|\to0.
\end{equation}
By Proposition~\ref{prop:endpoint-green},
\begin{equation}\label{eq:av-expansion}
        a v_a=d_a+e_a,
        \qquad
        d_a(y)=\frac{\ip{Q_a^\sigma}{y}}{\abs y^3},
        \qquad
        \norm{e_a}_{W^{1,p}(\Omega)}\to0,
\end{equation}
where
\begin{equation}\label{eq:Qa-sublog}
        \abs{Q_a^\sigma}=o(\log a).
\end{equation}
Consequently
\begin{equation}\label{eq:va-size}
        \norm{v_a}_{W^{1,p}(\Omega)}
        \le C\frac{\abs{Q_a^\sigma}}{a}+o(a^{-1}).
\end{equation}
For large $a$, Proposition~\ref{prop:D-quadratic} applies and gives
$$
        D(a)=\Dfun(g_a,u_a)=\Dfun(\delta,\rhoE)+\LD(k_a,v_a)+\mathcal R(k_a,v_a).
$$
Lemma~\ref{lem:euclidean-calibration} gives $\Dfun(\delta,\rhoE)=0$.  Multiplying by $a$,
\begin{equation}\label{eq:aD-split}
        aD(a)=\LD(ak_a,av_a)+a\mathcal R(k_a,v_a).
\end{equation}
By the linearity and continuity of $\LD$ on $C^0(\Omega)\times W^{1,p}(\Omega)$,
$$
        \LD(ak_a,av_a)=\LD(ak_a,0)+\LD(0,d_a)+\LD(0,e_a).
$$
The first term tends to zero by \eqref{eq:ak-zero}; the second is zero by Lemma~\ref{lem:dipole-kernel}; the third tends to zero because $e_a\to0$ in $W^{1,p}$.  Hence
\begin{equation}\label{eq:linear-zero}
        \LD(ak_a,av_a)\to0.
\end{equation}
For the remainder, \eqref{eq:D-rem-bound}, \eqref{eq:ak-zero}, and \eqref{eq:va-size} imply
$$
\begin{aligned}
        a\abs{\mathcal R(k_a,v_a)}
        &\le Ca\left(\norm{k_a}_{C^0}+\norm{v_a}_{W^{1,p}}\right)^2              \\
        &\le Ca\left(o(a^{-1})+C\frac{\abs{Q_a^\sigma}}{a}+o(a^{-1})\right)^2           \\
        &\le C\frac{\abs{Q_a^\sigma}^2}{a}+o(1)
        =o\left(\frac{(\log a)^2}{a}\right)+o(1)=o(1).
\end{aligned}
$$
Combining this with \eqref{eq:aD-split} and \eqref{eq:linear-zero} proves \eqref{eq:aD-zero}.
\end{proof}

\section{Proof of Theorem \ref{thm:intro-main}}
\label{sec:conclusion}

\begin{proof}[Proof of Theorem \ref{thm:intro-main}]
By Proposition~\ref{prop:aD-zero}, $aD(a)\to0$ along dyadic $a\to\infty$.  Since $D(a)\ge0$ and $aD(a)$ is non-decreasing, for every fixed $a_0>0$ and every dyadic $a_j\ge a_0$,
$$
        0\le a_0D(a_0)\le a_jD(a_j).
$$
Letting $j\to\infty$ gives $a_0D(a_0)=0$.  Since $a_0$ is arbitrary, $D\equiv0$.

By Proposition~\ref{prop:D-zero-F-zero}, the right-continuous monotone representative of $F$ vanishes identically.  Proposition~\ref{prop:F-rigidity-appendix} then implies that $(\R^3,g)$ is isometric to Euclidean space.
\end{proof}

\appendix

\section{Elementary dyadic estimates}
\label{app:cesaro}

\begin{lemma}\label{lem:cesaro}
Let $m_j\ge0$ for $j\ge0$ and suppose that $m_j\to0$.  Then
$$
        \sum_{j=0}^N m_j=o(N),
        \qquad
        2^{-N}\sum_{j=0}^N2^jm_j=o(1),
        \qquad
        \sum_{\ell\ge0}4^{-\ell}m_{N+\ell}=o(1).
$$
\end{lemma}

\begin{proof}
The first statement is a fundamental Cesaro's theorem.  For the second, fix $\eps>0$ and choose $J$ such that $m_j\le\eps$ for $j\ge J$.  Then
$$
        2^{-N}\sum_{j=0}^N2^jm_j
        \le 2^{-N}\sum_{j<J}2^jm_j
        +\eps\,2^{-N}\sum_{J\le j\le N}2^j
        \le o(1)+2\eps.
$$
The third follows from
$$
        \sum_{\ell\ge0}4^{-\ell}m_{N+\ell}
        \le \left(\sup_{j\ge N}m_j\right)\sum_{\ell\ge0}4^{-\ell}\to0.
$$
\end{proof}

\section{Perturbative elliptic upgrade}
\label{app:elliptic}

This appendix proves the elliptic estimate used in Propositions~\ref{prop:first-W1p} and~\ref{prop:endpoint-green}.  The applications in the paper involve smooth functions on fixed annuli; accordingly the estimates below are stated for smooth solutions.  This avoids any unnecessary weak-solution approximation issue.  The constants depend only on the indicated exponents and the inclusion of domains, not on derivatives of the coefficients.

\begin{lemma}[Constant-coefficient local estimate]\label{lem:laplace-local-Lq}
Let $1\le q\le p$, $3<p<\infty$, and let $B_r\Subset B_R\subset\R^3$.  Suppose $v$ is smooth, $G\in L^p(B_R;\R^3)$, and
$$
        \Delta v=\divE G
$$
in $B_R$.  Then
\begin{equation}\label{eq:laplace-local-Lq}
        \norm{\nabla v}_{L^p(B_r)}
        \le C\left((R-r)^{-\gamma}\norm{v}_{L^q(B_R)}
        +\norm{G}_{L^p(B_R)}\right),
        \qquad
        \gamma=1+3/q-3/p,
\end{equation}
where $C$ depends only on $p,q$.
\end{lemma}

\begin{proof}
We first prove the normalized estimate for $B_{1/2}\Subset B_1$.  Let $z\in W^{1,p}_0(B_1)$ solve
$$
        \Delta z=\divE G,
        \qquad z|_{\partial B_1}=0.
$$
The Calderon--Zygmund estimate for the Dirichlet Laplacian, see for example \cite[Chapter~9]{GT}, gives
$$
        \norm{\nabla z}_{L^p(B_1)}\le C_p\norm{G}_{L^p(B_1)}.
$$
Since $p>3$, Sobolev--Poincare inequality implies
$$
        \norm{z}_{L^q(B_1)}\le C_{p,q}\norm{G}_{L^p(B_1)}.
$$
The difference $h=v-z$ is harmonic in $B_1$.  Harmonic interior estimates give
$$
        \norm{\nabla h}_{L^p(B_{1/2})}
        \le C_{p,q}\norm{h}_{L^q(B_1)}.
$$
Combining the last three inequalities yields
$$
        \norm{\nabla v}_{L^p(B_{1/2})}
        \le C\left(\norm{v}_{L^q(B_1)}+\norm{G}_{L^p(B_1)}\right).
$$
After scaling, the same estimate holds on $B_\rho(x_0)\Subset B_{2\rho}(x_0)$ with the factor $\rho^{-\gamma}$ in front of $\|v\|_{L^q}$, where
$$
        \gamma=1+3/q-3/p>0.
$$
For a general concentric pair $B_r\Subset B_R$, set $\rho=(R-r)/4$ and cover $B_r$ by balls $B_\rho(x_\ell)$ with centers in $B_r$ so that the enlarged balls $B_{2\rho}(x_\ell)$ have uniformly bounded overlap.  Since $B_{2\rho}(x_\ell)\subset B_R$, the scaled estimate applies on each $B_\rho(x_\ell)\Subset B_{2\rho}(x_\ell)$.  Summing the $p$-th powers, using bounded overlap for the $G$ term, and using $q\le p$ together with bounded overlap for the $v$ term, gives
$$
        \norm{\nabla v}_{L^p(B_r)}
        \le C\left(\rho^{-\gamma}\norm{v}_{L^q(B_R)}+\norm{G}_{L^p(B_R)}\right).
$$
Since $\rho$ is comparable to $R-r$, this is \eqref{eq:laplace-local-Lq}.  The exponent $\gamma$ is exactly the scaling loss from the $L^q$ norm of $v$ to the $L^p$ norm of $\nabla v$.
\end{proof}

\begin{lemma}[Hole-filling iteration]\label{lem:hole-filling}
Let $0<r_0<r_1$ and let $\Phi:[r_0,r_1]\to[0,\infty)$ be bounded.  Suppose that for some $0<\theta<1$, $A_0\ge0$, and $\gamma>0$,
\begin{equation}\label{eq:hole-filling-assumption}
        \Phi(s)\le \theta\Phi(t)+\frac{A_0}{(t-s)^\gamma}
        \qquad\text{whenever }r_0\le s<t\le r_1.
\end{equation}
Then
\begin{equation}\label{eq:hole-filling-conclusion}
        \Phi(r_0)\le C_{\theta,\gamma}\frac{A_0}{(r_1-r_0)^\gamma}.
\end{equation}
\end{lemma}

\begin{proof}
Choose $\tau\in(0,1)$ sufficiently close to $1$ so that
\begin{equation}\label{eq:tau-choice}
        \theta\tau^{-\gamma}<1.
\end{equation}
Set
$$
        r_k=r_0+(1-\tau^k)(r_1-r_0),
        \qquad k=0,1,2,\ldots.
$$
Then $r_{k+1}-r_k=(1-\tau)\tau^k(r_1-r_0)$.  Applying \eqref{eq:hole-filling-assumption} with $s=r_k$ and $t=r_{k+1}$ gives
$$
        \Phi(r_k)
        \le \theta\Phi(r_{k+1})
        +A_0(1-\tau)^{-\gamma}\tau^{-k\gamma}(r_1-r_0)^{-\gamma}.
$$
Iterating from $k=0$ to $N-1$ gives
$$
\begin{aligned}
        \Phi(r_0)
        &\le \theta^N\Phi(r_N)
        +A_0(1-\tau)^{-\gamma}(r_1-r_0)^{-\gamma}
          \sum_{k=0}^{N-1}\theta^k\tau^{-k\gamma}.
\end{aligned}
$$
Since $\Phi$ is bounded and $\theta^N\to0$, the terminal term tends to zero.  The series converges by \eqref{eq:tau-choice}.  Letting $N\to\infty$ proves the estimate, with
$$
        C_{\theta,\gamma}=\frac{(1-\tau)^{-\gamma}}{1-\theta\tau^{-\gamma}}.
$$
\end{proof}

\begin{proposition}[Local perturbative $L^q$-to-$W^{1,p}$ estimate]\label{prop:elliptic-perturbative-full}
Let $\Omega'\Subset\Omega\subset\R^3$, let $1\le q\le p$, and let $3<p<\infty$.  There are constants $\eps_{p,q,\Omega',\Omega}>0$ and $C<\infty$ with the following property.  Suppose $A\in L^\infty(\Omega)$ satisfies
$$
        \norm{A-I}_{L^\infty(\Omega)}\le \eps_{p,q,\Omega',\Omega},
$$
and suppose $v$ is smooth and solves
$$
        \divE(A\nabla v)=\divE Y
$$
in $\Omega$, with $Y\in L^p(\Omega;\R^3)$.  Then
\begin{equation}\label{eq:perturbative-full-estimate}
        \norm{\nabla v}_{L^p(\Omega')}
        \le C\left(\norm{v}_{L^q(\Omega)}+\norm{Y}_{L^p(\Omega)}\right).
\end{equation}
\end{proposition}

\begin{proof}
Write $C_A=I-A$.  The equation is equivalent to
\begin{equation}\label{eq:laplace-perturbed}
        \Delta v=\divE(C_A\nabla v+Y).
\end{equation}
Fix a ball $B_{4r}(x_0)\Subset\Omega$.  For any $r\le s<t\le4r$, Lemma~\ref{lem:laplace-local-Lq}, applied on $B_s(x_0)\Subset B_t(x_0)$ with
$$
        G=C_A\nabla v+Y,
$$
gives
\begin{equation}\label{eq:local-pre-iteration}
        \Phi(s)
        \le C_0\eps\Phi(t)
        +C_0\norm{Y}_{L^p(B_{4r})}
        +C_0(t-s)^{-\gamma}\norm{v}_{L^q(B_{4r})},
\end{equation}
where
$$
        \Phi(\rho)=\norm{\nabla v}_{L^p(B_\rho(x_0))},
        \qquad
        \eps=\norm{A-I}_{L^\infty(\Omega)},
        \qquad
        \gamma=1+3/q-3/p.
$$
Since $t-s\le3r$, the $Y$ term can be absorbed into the same form by setting
$$
        A_1=C_0\norm{v}_{L^q(B_{4r})}+C_0(3r)^\gamma\norm{Y}_{L^p(B_{4r})}.
$$
Then \eqref{eq:local-pre-iteration} implies
$$
        \Phi(s)\le \theta\Phi(t)+A_1(t-s)^{-\gamma},
        \qquad \theta=C_0\eps.
$$
Choose $\eps$ so small that $\theta<1$.  Lemma~\ref{lem:hole-filling} applied on $[r,4r]$ yields
\begin{equation}\label{eq:ball-perturbative}
        \norm{\nabla v}_{L^p(B_r(x_0))}
        \le C\left(r^{-\gamma}\norm{v}_{L^q(B_{4r}(x_0))}
        +\norm{Y}_{L^p(B_{4r}(x_0))}\right).
\end{equation}
A finite cover of $\Omega'$ by balls $B_r(x_\ell)$ with $B_{4r}(x_\ell)\Subset\Omega$ and summation of \eqref{eq:ball-perturbative} give \eqref{eq:perturbative-full-estimate}.
\end{proof}

\begin{proposition}[Sequence upgrade from $L^q$ to $W^{1,p}$]\label{prop:elliptic-sequence}
Let $\Omega'\Subset\Omega\subset\R^3$, let $1\le q\le p$, and let $3<p<\infty$.  Suppose $A_i\in L^\infty(\Omega)$ satisfy
$$
        A_i\to I\qquad\text{in }L^\infty(\Omega),
$$
and suppose $v_i$ are smooth solutions of
$$
        \divE(A_i\nabla v_i)=\divE Y_i
$$
in $\Omega$.  If
$$
        \norm{v_i}_{L^q(\Omega)}\to0,
        \qquad
        \norm{Y_i}_{L^\infty(\Omega)}\to0,
$$
then
\begin{equation}\label{eq:sequence-upgrade-conclusion}
        \norm{v_i}_{W^{1,p}(\Omega')}\to0.
\end{equation}
\end{proposition}

\begin{proof}
Because $\Omega$ has finite measure, $Y_i\to0$ in $L^p(\Omega)$.  For all sufficiently large $i$, the smallness hypothesis in Proposition~\ref{prop:elliptic-perturbative-full} holds.  Thus
$$
        \norm{\nabla v_i}_{L^p(\Omega')}
        \le C\left(\norm{v_i}_{L^q(\Omega)}+
        \norm{Y_i}_{L^p(\Omega)}\right)\to0.
$$
To control the zeroth-order term, choose a smooth domain $\Omega''$ with $\Omega'\Subset\Omega''\Subset\Omega$.  Applying the gradient estimate with $\Omega''$ in place of $\Omega'$ gives $\nabla v_i\to0$ in $L^p(\Omega'')$.  Poincare's inequality yields
$$
        \norm{v_i-(v_i)_{\Omega''}}_{L^p(\Omega'')}
        \le C\norm{\nabla v_i}_{L^p(\Omega'')}\to0.
$$
Moreover,
$$
        |(v_i)_{\Omega''}|
        \le C\norm{v_i}_{L^q(\Omega'')}
        \le C\norm{v_i}_{L^q(\Omega)}\to0.
$$
Hence $v_i\to0$ in $L^p(\Omega'')$, and therefore in $L^p(\Omega')$.  Together with the gradient convergence this proves \eqref{eq:sequence-upgrade-conclusion}.
\end{proof}

\section{\texorpdfstring{Endpoint validity of the $F/D$ machinery}{Endpoint validity of the F/D machinery}}
\label{app:FD-endpoint}

This appendix records the functional facts used in Sections~\ref{sec:D}--\ref{sec:conclusion}.  The purpose is to separate the issue of asymptotic decay from the level-set identities.  Throughout this appendix, $g$ is a smooth complete metric on $\R^3$, $R_g\ge0$, and $u$ is the positive Green function with pole at the origin, normalized by
$$
        \int_{\{u=s\}} |\nabla_g u|\,dA_g=4\pi
$$
for a.e. regular level $s>0$.  We assume only that $u\to0$ at infinity and that the level sets in question are compact.  In the endpoint situation these properties follow from the two-sided Green estimate \eqref{eq:green-two-sided}.

All level-set integrals below are first understood on regular levels.  Equivalently, in the Agostiniani--Mazzieri--Oronzio formulation one integrates over the regular part $\Sigma_t\setminus\Crit(u)$ and then passes to the locally absolutely continuous representative of the resulting function.  Whenever monotonicity is used, $F$ denotes this right-continuous non-decreasing representative.

\begin{lemma}[Properness and connectedness of regular levels]\label{lem:level-connected}
For each regular value $s>0$, the level set
$$
        \Gamma_s=\{u=s\}
$$
is compact and connected.
\end{lemma}

\begin{proof}
Compactness follows from $u(x)\to0$ as $|x|\to\infty$ and $u(x)\to\infty$ at the pole.  Let $\Gamma$ be a connected component of $\Gamma_s$.  By the Jordan--Brouwer separation theorem, $\R^3\setminus\Gamma$ has one bounded component, denoted $\Omega_\Gamma$, whose boundary is $\Gamma$.  If $\Omega_\Gamma$ did not contain the pole, then $u$ would be harmonic on $\Omega_\Gamma$ and equal to $s$ on $\partial\Omega_\Gamma$; the maximum principle would force $u\equiv s$ on $\Omega_\Gamma$, contradicting that $s$ is a regular value.  Hence the pole lies in $\Omega_\Gamma$ for every component $\Gamma$.

Suppose two distinct components $\Gamma_1$ and $\Gamma_2$ existed.  Since both bounded regions contain the pole and the components are disjoint, the corresponding bounded regions are nested after relabeling, say
$$
        \Omega_{\Gamma_1}\Subset \Omega_{\Gamma_2}.
$$
The annular region $\Omega_{\Gamma_2}\setminus\overline{\Omega_{\Gamma_1}}$ contains no pole.  The function $u$ is harmonic there and takes the same boundary value $s$ on both boundary components.  The maximum principle again gives $u\equiv s$ on this annular region, impossible for a regular level.  Therefore only one component can occur, and $\Gamma_s$ is connected.
\end{proof}

\begin{proposition}[The $F$-functional in the present normalization]\label{prop:F-functional-facts}
For regular $t>0$, set
$$
        \Sigma_t=\Gamma_{1/t}=\{u=1/t\}.
$$
Define
\begin{equation}\label{eq:F-appendix-def}
        F(t)=4\pi t-t^2\int_{\Sigma_t}H_g|\nabla_g u|\,dA_g
        +t^3\int_{\Sigma_t}|\nabla_g u|^2\,dA_g,
\end{equation}
where
$$
        H_g=-\operatorname{div}_g\left(\frac{\nabla_g u}{|\nabla_g u|}\right).
$$
Let $\II$ denote the second fundamental form of $\Sigma_t$ and $\mathring{\II}$ its trace-free part.  Then $F$ has a locally absolutely continuous representative, and for a.e. $t$,
\begin{equation}\label{eq:F-prime-appendix}
\begin{aligned}
        F'(t)=4\pi+
        \int_{\Sigma_t}&\left(
        -\frac12R_{\Sigma_t}
        +\frac{|\nabla_{\Sigma_t}|\nabla_g u||^2}{|\nabla_g u|^2}
        +\frac12R_g
        +\frac12|\mathring{\II}|^2 
        +\frac34\left(\frac{2|\nabla_g u|}{u}-H_g\right)^2
        \right)dA_g.
\end{aligned}
\end{equation}
Consequently, if $R_g\ge0$, then $F$ is non-decreasing and $F\ge0$.
\end{proposition}

\begin{proof}
Formula \eqref{eq:F-prime-appendix} is the Agostiniani--Mazzieri--Oronzio monotonicity identity \cite[Theorem~1.1, formula~(1.3)]{AMO}, written in the convention $u_{\mathrm{AMO}}=1-u$ and with the level parameter $t=1/u$.  Its proof is local on regular level sets and uses only the harmonicity of $u$, the Bochner identity, the first variation of level-set geometry, and the flux normalization; it does not require any asymptotic expansion at infinity.

By Lemma~\ref{lem:level-connected}, each regular level is connected.  Hence Gauss--Bonnet gives
$$
        \int_{\Sigma_t} R_{\Sigma_t}\,dA_g=4\pi\chi(\Sigma_t)\le8\pi.
$$
Therefore
$$
        4\pi-\frac12\int_{\Sigma_t}R_{\Sigma_t}\,dA_g\ge0.
$$
All other terms in \eqref{eq:F-prime-appendix} are non-negative when $R_g\ge0$.  Thus $F'\ge0$ a.e., so the right-continuous representative of $F$ is non-decreasing.

Near the Green pole the metric is smooth and $u$ has the usual Euclidean singular expansion in normal coordinates.  Substituting this local expansion into \eqref{eq:F-appendix-def} gives
$$
        \lim_{t\downarrow0}F(t)=0.
$$
Since $F$ is non-decreasing, $F(t)\ge0$ for all $t>0$.
\end{proof}

\begin{proposition}[Derivative of the Dirichlet level quantity]\label{prop:I-derivative}
Let
$$
        I(t)=\int_{\Sigma_t}|\nabla_g u|^2\,dA_g.
$$
Then $I$ is locally absolutely continuous and, for a.e. $t$,
\begin{equation}\label{eq:I-prime-appendix}
        I'(t)=-t^{-2}\int_{\Sigma_t}H_g|\nabla_g u|\,dA_g.
\end{equation}
\end{proposition}

\begin{proof}
This is the first variation formula for integrals over the level sets $\{u=1/t\}$.  If $s=1/t$, the normal speed of the level $\{u=s\}$ under variation of $t$ is $-t^{-2}/|\nabla_g u|$.  Applying the standard first variation formula to $|\nabla_g u|^2$ and using $\Delta_g u=0$ gives \eqref{eq:I-prime-appendix}.  As with the $F$-identity, this is a local computation on compact regular level sets.
\end{proof}

\begin{proposition}[The $D$-functional identities]\label{prop:D-functional-facts}
Let $\psi\in C_c^\infty((0,1))$, $\psi\ge0$, $\int_0^1\psi=1$, and define $\Efun,D,c_\psi,\phi$ as in \eqref{eq:E-def}--\eqref{eq:phi-def}.  Then
$$
        D(a)\ge0,
        \qquad
        a\mapsto aD(a)\quad\text{is non-decreasing},
$$
and
\begin{equation}\label{eq:D-appendix-unscaled}
        D(a)=c_\psi+
        \int_{\R^3}\phi(au(x))|\nabla_g u(x)|^3a^3\,dv_g(x).
\end{equation}
After the rescaling $u_a(y)=au(ay)$ and $g_a(y)=g(ay)$,
\begin{equation}\label{eq:D-appendix-rescaled}
        D(a)=c_\psi+
        \int_{\R^3}\phi(u_a(y))|\nabla_{g_a}u_a(y)|^3\,dv_{g_a}(y).
\end{equation}
\end{proposition}

\begin{proof}
Non-negativity and monotonicity are the $D$-functional facts used in \cite[Propositions~14 and~15]{MYrigidity}.  Non-negativity is immediate from $F\ge0$.  For monotonicity, write $s=a\vartheta$ and then $t=(1+\vartheta)a\tau$ in the defining integral:
$$
        aD(a)=\int_0^1\frac{\psi(\vartheta)}{(1+\vartheta)^2}
        \left(\int_1^2\frac{F((1+\vartheta)a\tau)}{\tau^3}\,d\tau\right)d\vartheta.
$$
Since $F$ is non-decreasing and $\psi\ge0$, this expression is non-decreasing in $a$.

Let $T=a+s$.  Using Proposition~\ref{prop:I-derivative},
$$
\begin{aligned}
        \int_T^{2T}\frac{F(t)}{t^3}\,dt
        &=\int_T^{2T}\left(\frac{4\pi}{t^2}
        -\frac1t\int_{\Sigma_t}H_g|\nabla_g u|\,dA_g
        +I(t)\right)dt \\
        &=\frac{2\pi}{T}+[tI(t)]_{t=T}^{2T} \\
        &=\frac{2\pi}{T}
        +\int_{\{u=1/(2T)\}}\frac{|\nabla_g u|^2}{u}\,dA_g
        -\int_{\{u=1/T\}}\frac{|\nabla_g u|^2}{u}\,dA_g.
\end{aligned}
$$
Integrating this identity against $\psi(s/a)\,ds$ and applying the coarea formula gives \eqref{eq:D-appendix-unscaled} with the stated $c_\psi$ and $\phi$.  The level bands involved are compact because $u\to0$ at infinity and $u\to\infty$ at the pole.  The change of variables $x=ay$ gives \eqref{eq:D-appendix-rescaled}.
\end{proof}

\begin{proposition}[$D\equiv0$ forces $F\equiv0$]\label{prop:D-zero-F-zero}
If $D(a)=0$ for every $a>0$, then the right-continuous monotone representative of $F$ vanishes identically.
\end{proposition}

\begin{proof}
Suppose that $F(t_0)>0$ for some $t_0>0$.  By monotonicity, $F(t)\ge F(t_0)>0$ for all $t\ge t_0$.  Choose an interval $I\Subset(0,1)$ on which $\psi>0$, choose $\vartheta_0\in I$, and set $a=t_0/(1+\vartheta_0)$.  For all $\vartheta$ in a smaller interval $I_0\Subset I$ containing $\vartheta_0$, the interval $((1+\vartheta)a,2(1+\vartheta)a)$ intersects $[t_0,\infty)$ in a set of positive measure.  Hence
$$
        \int_{(1+\vartheta)a}^{2(1+\vartheta)a}\frac{F(t)}{t^3}\,dt>0
        \qquad\text{for }\vartheta\in I_0.
$$
Since $\psi>0$ on $I_0$, the representation of $D(a)$ gives $D(a)>0$, a contradiction.
\end{proof}

\begin{proposition}[AMO $F$-rigidity]\label{prop:F-rigidity-appendix}
If $R_g\ge0$ and the right-continuous monotone representative of $F$ satisfies $F\equiv0$, then $(\R^3,g)$ is isometric to Euclidean space.
\end{proposition}

\begin{proof}
We verify the hypotheses in \cite[Theorem~1.1 and Corollary~1.3]{AMO}.  The manifold is smooth, complete, noncompact, and has $R_g\ge0$ by assumption.  The endpoint decay \eqref{eq:c0-decay} makes $g$ uniformly equivalent to the Euclidean metric outside a compact set, while smoothness gives uniform equivalence on the compact part.  In particular the Green function with Euclidean two-sided bounds constructed in Section~\ref{sec:setup} exists; equivalently, $(\R^3,g)$ is nonparabolic.  The topological condition is also satisfied because
$$
        H_2(\R^3;\mathbb Z)=0.
$$
Our Green function $u$ is the minimal positive Green function normalized to have flux $4\pi$ over regular levels, so if $G_o$ denotes the minimal positive Green function in the normalization used by \cite{AMO}, then $u=4\pi G_o$.  Consequently
$$
        u_{\mathrm{AMO}}=1-4\pi G_o=1-u
$$
is exactly the maximal distributional solution used in \cite[Theorem~1.1]{AMO}; it tends to $1$ at infinity and has the prescribed pole at the origin, with the sign convention of \cite{AMO}.  The level-set integrals in Proposition~\ref{prop:F-functional-facts} are the same as the AMO integrals over the regular part of the levels after the change of parameter $t=1/u$ and the convention $u_{\mathrm{AMO}}=1-u$.

Their Corollary~1.3 states, under these hypotheses, that if the monotone quantity of Theorem~1.1 has limit zero at infinity, then the manifold is isometric to $(\R^3,g_{\R^3})$.  In the present situation $F\equiv0$, so the required limit is zero.  Therefore \cite[Corollary~1.3]{AMO} applies and gives flatness.
\end{proof}

\end{document}